\documentclass[11pt]{article} 
\usepackage[latin1]{inputenc}
\usepackage{amsmath,amssymb}
 \usepackage{xcolor}
\usepackage{multicol}
\usepackage{floatrow}
\usepackage{latexsym,epsfig}
\usepackage{amssymb,amsmath,amsfonts,amscd,amsthm}
\usepackage{exscale}
\usepackage{ifthen}
\usepackage{longtable}
\usepackage{subfigure}
 \usepackage{cite}
\usepackage{lscape}
\usepackage{blindtext}
\usepackage{soul}
\usepackage{mathtools}
 \usepackage{indentfirst}
\usepackage{dirtytalk}
\usepackage{ upgreek }

\usepackage{ulem}
\usepackage{enumerate}
\usepackage{ bbold }

\usepackage{cases}
\usepackage{ mathrsfs }
\usepackage{ latexsym }
\usepackage{setspace}
 \newtheorem{theo}{\textbf{Theorem}\ }
[section]
\newtheorem{lem}[theo]{\textbf{Lemma}\ }
\newtheorem{coro}[theo]{Corollary\ }

\newtheorem{prop}[theo]{\textbf{Proposition}\ }

\newtheorem{exam}[theo]{Example\ }

\newtheorem{remark}[theo]{\textbf{Remark}\ }

\def \x\ {\mathbb{X}}

	 				\definecolor{amber}{rgb}{1.0, 0.75, 0.0}
	\definecolor{airforceblue}{rgb}{0.36, 0.54, 0.66}
		\definecolor{alizarin}{rgb}{0.82, 0.1, 0.26}
	\definecolor{amaranth}{rgb}{0.9, 0.17, 0.31}	
	\definecolor{amber(sae/ece)}{rgb}{1.0, 0.49, 0.0}	
	\definecolor{americanrose}{rgb}{1.0, 0.01, 0.24}	
	\definecolor{amethyst}{rgb}{0.6, 0.4, 0.8}	
	\definecolor{antiquebrass}{rgb}{0.8, 0.58, 0.46}	
	\definecolor{ao(english)}{rgb}{0.0, 0.5, 0.0}	
	\definecolor{apricot}{rgb}{0.98, 0.81, 0.69}	
	\definecolor{aquamarine}{rgb}{0.5, 1.0, 0.83}	
	\definecolor{armygreen}{rgb}{0.29, 0.33, 0.13}
	\definecolor{arsenic}{rgb}{0.23, 0.27, 0.29}	
	\definecolor{asparagus}{rgb}{0.53, 0.66, 0.42}	
	\definecolor{auburn}{rgb}{0.43, 0.21, 0.1}	
	\definecolor{awesome}{rgb}{1.0, 0.13, 0.32}	
	\definecolor{ballblue}{rgb}{0.13, 0.67, 0.8}	
	\definecolor{bananayellow}{rgb}{1.0, 0.88, 0.21}	
	\definecolor{bittersweet}{rgb}{1.0, 0.44, 0.37}	
	\definecolor{bleudefrance}{rgb}{0.19, 0.55, 0.91}
	\definecolor{blush}{rgb}{0.87, 0.36, 0.51}	
	\definecolor{bostonuniversityred}{rgb}{0.8, 0.0, 0.0}	
	\definecolor{brickred}{rgb}{0.8, 0.25, 0.33}	
\definecolor{brightmaroon}{rgb}{0.76, 0.13, 0.28}	
	\definecolor{brilliantrose}{rgb}{1.0, 0.33, 0.64}
	\definecolor{britishracinggreen}{rgb}{0.0, 0.26, 0.15}	
	\definecolor{bulgarianrose}{rgb}{0.28, 0.02, 0.03}	
	\definecolor{burgundy}{rgb}{0.5, 0.0, 0.13}	
	\definecolor{burntorange}{rgb}{0.8, 0.33, 0.0}	
	\definecolor{byzantium}{rgb}{0.44, 0.16, 0.39}			
	\definecolor{candyapplered}{rgb}{1.0, 0.03, 0.0}
	\definecolor{capri}{rgb}{0.0, 0.75, 1.0}	
	\definecolor{carminered}{rgb}{1.0, 0.0, 0.22}	
	\definecolor{charcoal}{rgb}{0.21, 0.27, 0.31}	
	\definecolor{coquelicot}{rgb}{1.0, 0.22, 0.0}	
	\definecolor{daffodil}{rgb}{1.0, 1.0, 0.19}	
  
 \numberwithin{equation}{section}
 
\begin{document}
 
 \title{The oscillating random walk on $\mathbb Z$} 
 \maketitle 

\begin{center}
{\bf D. T. Vo}
$^($\footnote{ 
 Institut Denis Poisson UMR 7013,  Universit\'e de Tours, Universit\'e d'Orl\'eans, CNRS, France. \\ Tran-Duy.Vo@lmpt.univ-tours.fr}$^)$  	 
     \end{center}

         \centerline{\bf  \small Abstract } 
 The paper is concerned with a new approach for the recurrence property of the oscillating process on $\mathbb{Z}$ in Kemperman's sense. In the case when the random walk is ascending on $\mathbb{Z}^-$ and descending on $\mathbb{Z}^+$, we determine the invariant measure of the embedded process of successive crossing times  and then prove a necessary and sufficient condition for recurrence. Finally, we make use of this result to show that the general oscillating process is recurrent under some H{\"o}lder-typed moment assumptions.

  \vspace{0.5cm}



\section{Introduction and notations}
\subsection{Introduction}
 In parallel with many studies of classical stochastic processes,  oscillating random walks, which was introduced systematically by \textsc{Kemperman}\cite{Kemperman}, have been found to be good models with several applications, see \cite{KeilsonServi} for instance. This paper deals with the homogeneous Markov chain $\mathcal{X}^{(\alpha)}=(X_{n}^{(\alpha)})_{n\geq 0}$ indexed by a parameter $\alpha \in [0,1]$ such that $X_0^{(\alpha)}=x_0$ with some fixed $x_0 \in \mathbb Z$ and for $n\geq 1$,
 \begin{align} \label{def_alpha}
     X_{n+1}^{(\alpha)}:=X_n^{(\alpha)}+ \left(\xi_{n+1}\mathbb 1_{\{X_n^{(\alpha)} \leq -1\}} + \eta_{n+1}\mathbb 1_{\{X_n^{(\alpha)}=0\}} + \xi_{n+1}'\mathbb 1_{\{X_n^{(\alpha)} \geq 1\}} \right),
 \end{align}  
where 
\begin{itemize}
    \item the $\xi_n, n\geq 1$, have common distribution $\mu$,
    \item the $\xi'_n, n\geq 1$, have common distribution $\mu'$,
    \item the $\eta_n, n\geq 1$, have common distribution given linearly by
    $$\mathbb P[\eta_n=y]:= \alpha \mu(y) +(1-\alpha) \mu'(y) \text{ for any } y \in \mathbb Z,$$
    \item  $(\xi_n, \xi'_n, \eta_n)_{n \geq 1}$ is a sequence of independent and identically distributed (abbreviate i.i.d.) random variables.
\end{itemize}

When we want to emphasize the dependence of $\mathcal{X}^{(\alpha)}$ on the distributions $\mu$ and $\mu'$, the process $(X_n^{(\alpha)})_{n\geq 0}$ is denoted by $\mathcal{X}^{(\alpha)}(\mu,\mu')$ instead $\mathcal{X}^{(\alpha)}$. 

It is merely to say that the excursion will be directed by $\mu$ (resp. $\mu'$) as long as the process stays on the negative (resp. positive) side and therefore, this present model is often called the \textit{oscillating random walk} with respect to (w.r.t.) zero level. The choice of zero is arbitrary and can be replaced by any fixed level. In case of $\alpha \in\{0;1\}$, we use the terminology \say{\textit{crossing}} to mean the point $0$ belongs to only one of the two half lines and it belongs to both if $0<\alpha<1$.  We are mainly interested in recurrence of this process on its essential (i.e.  maximal irreducible)
classes. 

The case $\mu=\mu'$ is well-treated by $\textsc{Mijatovi\'c}$ and \textsc{Vysotsky}\cite{MijaVysot}, except providing detailed illustrations of the trajectory. The highlight result is an invariant measure (and a probability) constructed in a probabilistic manner and under the additional assumption of (topological) recurrence of the chain, it is up to a multiplicative constant finite invariant measure. For the sake of completeness, we will study all irreducible classes of $\mathcal{X}^{(\alpha)}$ from the simple case when $\xi_n\geq 0$ and $\xi_n'\leq 0$ to the general one. Appropriately refining the formula in \cite{MijaVysot}, we obtain the exact invariant measure and then apply the idea of \textsc{Knight}\cite{Knight78} to get its discrete version, see Section $2$.   

   Section $3$ is devoted to such an important sub-process of $\mathcal{X}^{(0)}$ evolving within a definite state space, whose elements are recorded at corresponding successive crossing times. A particular interest will be put on the structure of (essential) irreducible classes, especially on $\mathcal{N}$, the set of isolated states which are impossible to reach from any opposite states in a single step.  Theorem \ref{condition} stipulates some mild conditions for $\mathcal{N}$ to be empty and also yields an expression for it.  In analogy to the approach for reflected random walks, we finally compute the invariant measure for the sub-process based on some arguments developed in the previous work of \textsc{Peign\'e} and \textsc{Woess}\cite{MarcWoess}. 
     
   The recurrence of the general oscillating random walk $\mathcal{X}^{(\alpha)}$ is dealt with in Section $4$.  Some powerful tools such as the \textsc{Kemperman}'s criterion \cite{Kemperman} (a divergent series represented in term of renewal functions) or the integral criterion of \textsc{Rogozin} and \textsc{Foss} \cite{RogozinFoss} (a transformation established with the help of Wiener-Hopf factorization) are mentioned for reference. We also furnish a new approach coming from the fact that the recurrence of the crossing sub-process implies to  the recurrence of the full process. To do this, we first show that the process of crossing is (positive) recurrent if the tail distribution condition
   $$\displaystyle \sum_{n = 0}^{+\infty} \mathbb P(\xi_1 >n)\,\mathbb P(\xi_1'<-n) < +\infty$$
  holds when $\xi_n \geq 0, \xi_n'\leq 0$. Moreover, it can be attained under the hypothesis $\mathbb{E}[\xi_1^p]<+\infty$ and $\mathbb{E}[(-\xi_1')^q] <+\infty$, where $p, q \in ]0,1[$ satisfying $p+q=1$. When jumps are generalized on $\mathbb Z$, there may have different possibilities, for instance, both $\xi_n$ and $\xi_n'$ are either drifted (positive and negative, respectively) or centered as well as mixed and by Theorem \ref{extend_condition}, we will address suitable conditions to each corresponding situation. 
   
\subsection{Notations}    
  Throughout this paper, we fix some  frequently used notations
  \begin{itemize}
      \item $S_{\mu}\, (\text{resp. } S_\mu'):$ the support of $\mu \, (\text{resp. } \mu')$.
      \item  $D\, (\text{resp. } D'): \text{ the maximum of }\mu \, (\text{resp. the minimum of } \mu')$.\\
     We adhere to the convention that $D=+\infty $ (resp. $D'=-\infty$) when $S_\mu$ (resp. $S_{\mu'}$) is unbounded from above (resp. from below).
      \item $d \,(\text{resp. }d'):$ the greatest common divisor of $S_\mu \, (\text{resp. } S_\mu')$. 
      
      \item $\mathbb Z^{+}/\mathbb Z^{-}:$ the set of positive/negative integers $(\text{and }\mathbb Z_0^+/\mathbb Z_0^- \text{ if } 0 \text{ is included})$. 
      
        \item $r_x$: the remainder of $x$ in the Euclidean division by $\delta$ (i.e. $0\leq r_x<\delta$).
      \end{itemize}
      
 Let us end this paragraph dedicated to notations by reminding that for any fixed $0 \leq \alpha \leq 1$, the  chain \eqref{def_alpha} is denoted by
     $ \mathcal{X}^{(\alpha)} (\mu, \mu')$ (and simply $\mathcal{X}^{(\alpha)} $ when there is no ambiguity on the choices of $\mu$ and $\mu'$).

\section{Irreducible classes and invariant measure of $ \mathcal{X}^{(0)} (\mu, \mu')$} 
It is easy to check that if $\mu=\mu'$ (in this case $\mathcal{X}^{(0)} (\mu, \mu')$ becomes an ordinary random walk on $\mathbb Z$ with the unique jump measure $\mu$)  then $d=d'$ and the irreducible classes of $\mathcal{X}^{(0)} (\mu, \mu')$ are the sets  $r+d\mathbb Z$ with $0\leq r <d$. In this section, we  describe the essential classes of  $\mathcal{X}^{(0)} (\mu, \mu')$ when $\mu\neq \mu'$. 
\subsection{ The chain $  \mathcal{X}^{(0)}(\mu, \mu')$ when  $S_\mu \subset \mathbb Z^+$ and $S_{\mu'} \subset \mathbb Z^-$} 

For any $x \in \mathbb Z$, let  ${\mathcal I}(x)$ be the irreducible class of $x$. It holds $ {\mathcal I}(0)\subset \{D',   \ldots, D-1\}$. Furthermore, for any starting point  $x$, after finitely many steps, the chain $\mathcal X ^{(0)} $ stays  for ever a.s. in the subset $\{D', \ldots, D-1\}$.

\begin{theo}\label{irreducibleclass} {  We suppose that $S_\mu \subset \mathbb Z^+$ and $S_{\mu'} \subset \mathbb Z^-$.}

$\bullet$ Assume first $D$ and $D'$ are finite. If  $d \wedge d'=\delta$,  then 

i) there exist $\delta$ irreducible essential classes 
\[
 \{D', \ldots, D-1\}\cap (r+\delta \mathbb Z) \quad with \quad 0\leq r <\delta;
\]
(in particular  the irreducible class of $0$ equals $\mathcal I(0)= \{D', \ldots, D-1\}\cap  \delta \mathbb Z$);

ii)
 if $x \geq D$ or $x<D'$  then $x$ is transient  and, after finitely many steps, reaches $\mathbb P$-a.s. the  essential class $ \{D', \ldots, D-1\}\cap (r_x+\delta   \mathbb Z)$.
 
 $\bullet$ If $D'=-\infty$ and $D$ is finite then
 
 i)  there exist $\delta$ irreducible essential classes 
 \[
 ]-\infty,  D-1]\cap (r+\delta \mathbb Z) \quad with \quad 0\leq r <\delta;
\]

ii)
 the $x \geq D$   are all  transient  and, after finitely many steps, reaches $\mathbb P$-a.s. the  essential class $ ]-\infty, \ldots, D-1]\cap (r_x+\delta \mathbb Z)$.
 
 (similar dual statement follows when $D'$ is finite and $D=+\infty$).
 


$\bullet$ If $D=+\infty$  and $D'=-\infty$  then
  there exist $\delta$ irreducible essential classes, which are all essential: 
 \[
 r+\delta \mathbb Z  \quad with \quad 0\leq r <\delta.
\]

\end{theo}

\begin{proof}
The proof of \textsc{Kemperman} based on the theory of semi-groups of $\mathbb Z$ entirely solved for the chain $ \mathcal{X}^{(\alpha)}, 0\leq \alpha \leq 1$ providing that $D=-D'=+\infty$ (see Remark \ref{ess_general}).
 Back to the current model, we will prove by induction, but let us first fix some notations.

Let $T$ be a finite subset of  $ S_\mu\cup S_{\mu'}$ s.t.  $ T\cap S_\mu \neq \emptyset $ and $T\cap S_{\mu'}\neq \emptyset$; without loss of generality, we assume $0\notin T$.

For any   $x \in \mathbb Z$, we  denote by  $\mathcal O_{T}(x)$ the ``{\it orbit of $x$  under $T$}'', that is the set  of sequences ${\bf x}= (x_i)_{i \geq 0}$ defined  by  induction   as follows: $x_0=x$ and, for any $i\geq 1$,  

 - if $x_i\leq -1$ then $x_{i+1}= x_i+s $ for some $s \in T \cap S_\mu$; 
 
  - if $x_i\geq 0$ then $x_{i+1}= x_i+s'$ for some $s' \in T\cap S_{\mu'}.$

Notice that all the   $x_i$ but finitely many  do belong to $\{ \min(S_{\mu'}),  \ldots, \max (S_\mu)-1\}$.

For any $x, y \in \mathbb Z$, we write $x\ {\stackrel{T}{\rightarrow}}\ y$ if there exists ${\bf x} \in \mathcal O_T(x)$ s.t. $x_0=x$ and $x_n= y$ for some $n\geq 0$. When there exists  no such sequence $\bf x$, we write  $x\ {\stackrel{T}{\not\rightarrow}}\ y$. 
The notation $x\ {\stackrel{T}{\leftrightarrow}}\ y$ means $x\ {\stackrel{T}{\rightarrow}}\ y$ and $y\ {\stackrel{T}{\rightarrow}}\ x$.

The relation  $ {\stackrel{T}{\leftrightarrow}} $ is an equivalence relation on $\mathbb Z$ whose classes are called  {\it $T$-irreducible classes.  }
The $T$-irreducible class of $x$ is denoted $\mathcal I_{T}(x)$. 

The relation $ {\stackrel{T}{\rightarrow}} $    induces a partial order relation on $\mathcal I_T$, denoted again $ {\stackrel{T}{\rightarrow}}$. The maximal irreducible classes for this relation are called $T${\it -essential}; a non essential irreducible class is said $T$-transient.

 We now describe ${\mathcal I}_T$, by induction on the cardinality of $T$.

 \noindent \underline{{\bf Step 1-} {\it Case when   $T= \{s, s'\}$ with $s \in \mathbb Z^{+}$ and $s'\in \mathbb Z^{-}$. }}

$\bullet$ {\it We assume first $s\wedge s'=1,  x=0$ and prove that $\{s',   \ldots, s-1\}$  is the unique $T$-essential  class. Furthermore, $\mathcal I_T(x)$ equals $\{x\}$ and is $T$-transient when  $x\geq s$ or $x<s'$.}

  Indeed,   for any  ${\bf \omega} = ( \omega_i)_{i \geq 0}  $ in $\mathcal O_T(0)$,
  the $\omega_i$ all belong to $\mathbb N s+\mathbb N s'$ and to $\{s', \ldots, s-1\}$. Hence, there exist $j>i\geq 1$ such that $\omega_j=\omega_i$. Since $\omega_{i}\neq \omega_{i+1}, \omega_{j-1}\neq \omega_j$ and $\omega_i\  {\stackrel{T}{\rightarrow}}  \ \omega_j=\omega_i$, there exists $k, \ell \geq 1$ such that $ks+\ell s' = 0$. The condition  $s\wedge s'=1$ yields $k=0$ mod$(s')$ and $\ell = 0$ mod$(s$), hence $k\geq |s'|$ and $\ell \geq s$. Consequently, the sub-orbit $\{\omega_i, \omega_{i+1}, \ldots, \omega_{j-1}\}$ contains at least $s+|s'|$ elements; since it is included in $\{s', \ldots, s-1\}$, it holds in fact  $\{\omega_i, \omega_{i+1}, \ldots, \omega_{j-1}\}=\{s', \ldots, s-1\}.$  This proves that $x\  {\stackrel{T}{\rightarrow}}  \ y$ for any $x, y \in \{s', \ldots, s-1\}$, hence $\{s', \ldots, s-1\} \subset \mathcal I_T(0)$. Eventually, $\{s', \ldots, s-1\}= \mathcal I_T(0)$ since the elements of all the orbits of $0$ remain in $\{s', \ldots, s-1\}$. This also implies that $\{s', \ldots, s-1\}$ is $T$-essential.

  As a consequence of this argument,  the orbit $\mathcal O_T(0)$  contains a unique sequence   $\omega$, which is  is periodic with period  $ \omega_0, \omega_1, \ldots, \omega_{s +|s'|-1}$  where $\omega_i \in \{s', \ldots, s-1\}$ and $\omega_i\neq \omega_j$ for any $0\leq i<j<  s +|s'|.$ We write for short $ \omega = \overline {\omega_0, \ldots, \omega_{s +|s'|-1}}$ and emphasize that  $\{\omega_0, \ldots, \omega_{s +|s'|-1}\}=\{s', \ldots, s-1\}$.

 Now, if $x\geq s$ or $x<s'$ and ${\bf x} = (x_i)_{i \geq 0} \in \mathcal O_T(x)$, then $x_i \neq x$ for any $i\geq 1$; in other words, $\mathcal I_T(x) = \{x\}$ and  $x$ is $T$-transient.

 $\bullet$ {\it Assume now $s\wedge s'= \delta \geq 2$. Then, the $T$-essential classes are $\{s', \ldots, s-1\}\cap (x+\delta \mathbb Z)$.  Furthermore, if $x\geq s$ or $x<s'$, then $\mathcal I(x)$ equals $\{x\}$ and is $T$-transient.}

In this case, the set  $\mathcal O_T(0)$  still contains   a unique sequence 
 $\omega= \overline {\omega_0, \ldots, \omega_{k-1}}$  with $k= {s+ |s'| \over \delta} $ and $\{\omega_0, \ldots, \omega_{ k-1}\}=\{s', \ldots, s-1\}\cap \delta \mathbb Z$. As a direct consequence, for any $x \in \mathbb Z$, the set $\mathcal O_T(x)$ is included in $x+\delta \mathbb Z$ and also contains a unique sequence ${\bf x}$, which is ultimately periodic with period $r_x+\omega_0, \ldots,  r_x+\omega_{k-1}$. The description of the $T$-irreducible classes follows immediately.  
 \vspace{3mm}
 
  \noindent \underline{{\bf Step 2-}\it   Case when $T\subset \mathbb Z$ is finite and satisfies $T\cap \mathbb Z^{-} \neq \emptyset$ and  $T\cap \mathbb Z^{+} \neq \emptyset$}.
 
 The proof is made by induction, from $T$ to   ${\bf T}:= T\cup \{t'\}$ with $t'\in  \mathbb Z^{-}$  ; the case when ${\bf T}:= T\cup \{t\}$ with ${t}\in  \mathbb Z^{+}$   is studied in the same way. By Step 1, the property is true for $T=\{s, s'\}, s>0$ and $s'<0$.
 
\noindent \underline{Hypothesis of induction:}  {\it
 Let $T$ be a  set of  non-zero integers s.t. $ T\cap \mathbb Z^{-} \neq \emptyset$  and  $ T\cap \mathbb Z^{+} \neq \emptyset$. We set $\delta_T:=\text{\rm gcd}(T)$ and denote $s$ (resp. $s'$) the largest (resp. smallest) element of $T$. We assume that
 
 - the $T$-essential classes are $\{s', \ldots, s-1\}\cap (r+\delta_T \mathbb Z)$ with $0\leq r<\delta_{ T}$;
 
 - when $x\geq s$ or $x<s'$, then  $\mathcal I_T(x)=\{x\}$ and $x$ is $T$-transient; furthermore, for any ${\bf x} = (x_i)_{i \geq 0} \in \mathcal O_T(x)$,  all $x_i$ but  finitely many  belong to  the $T$-essential class $\{s', \ldots, s-1\}\cap (x+\delta_T \mathbb Z)$.

\noindent \underline{Conclusion:} For $t' \in  \mathbb Z^{-}$ then the same property holds for ${\bf T}= T\cup\{t'\}$. In other words, setting $\delta_{\bf T}:= \text{\rm gcd}({\bf T})$ and $m_{\bf T}:= \min(s', t')$,

 - the ${\bf T}$-essential classes are $\{m_{\bf T}, \ldots, s-1\}\cap (r +\delta_{\bf T} \mathbb Z)$ with $0\leq r <\delta_{\bf T}$;
 
 -  if $x \geq s$ or $x<m_{\bf T}$, then $\mathcal I_{\bf T}(x)= \{x\}$ and $x$ is ${\bf T}$-transient; furthermore, for any ${\bf x} = (x_i)_{i \geq 0} \in \mathcal O_{\bf T}(x)$,  all $x_i$ but  finitely many  belong to  the $\bf T$-essential class $\{m_{\bf T}, \ldots, s-1\}\cap (x+\delta_{\bf T} \mathbb Z)$.
 }

The same argument as in Step 1  works to deduce the case $\delta_{\bf T}\geq 2$ from the case $\delta_{\bf T}=1$; thus, we only consider the case $\delta_{\bf T}=1$. 

Notice that $\{m_{\textbf{T}},\dotsi,s-1\}$ is absorbing; in other words, for any $x \in \mathbb{Z}$ and any $\bf{x}=(x_i)_{i \geq 0} \in \mathcal{O}_{\bf{T}}(x)$, all $x_i$ but finitely many belong to this set. In particular, $x$ is transient when $x\geq s$ or $x<m_{\bf T}$. It thus remains to check that $\{m_{\bf T},\dotsi,s-1\}$ is irreducible. 
 
Let us first prove that 
\begin{equation}\label{zlibekjf}
\{s', \ldots, s-1\} \ {\stackrel{\bf T}{\rightarrow}}\ 
\{s', \ldots, s-1\}
\end{equation}
For any $x \in \{s', \ldots, s-1\} $, we choose $\ell_x\geq 0$ s.t. $x+t'+\ell_x s\in\{0, \ldots, s-1\}$ and notice that $x\ {\stackrel{\bf T}{\rightarrow}}\ x+t'+\ell_x s$. Now, for any ${\bf y} = (y_i)_{i\geq 0} \in \mathcal O_T(x+t'+\ell_x s)$, all $y_i$ but finitely  many  belong to the $T$-essential class $\mathcal I_T(x+t')=\{s, \ldots, s'-1\}\cap (x+t'+\delta_T \mathbb Z)$;  thus, $\{x\}\ {\stackrel{\bf T}{\rightarrow}}\ \{s', \dots, s-1\} \cap (x+t'+\delta_T \mathbb Z).$  Reiterating the argument, we get, for $k\geq 1$  \begin{equation}\label{libve}
 \{x\}\ {\stackrel{\bf T}{\rightarrow}}\ \{s', \dots, s-1\} \cap (x+kt'+\delta_T \mathbb Z).
\end{equation}
Since $\text{\rm gcd}(t', \delta_T)=\delta_{\bf T}=1$, the class of $t'\ \text{\rm mod}(\delta_T)$ generates $\mathbb Z/\delta_T\mathbb Z$, so 
 $$
 \bigcup_{0\leq k<\delta_T}\{s', \ldots, s-1\}\cap (x+kt'+\delta_T \mathbb Z)=\{s', \ldots, s-1\}.
 $$
This yields immediately $\{x\} \ {\stackrel{\bf T}{\rightarrow}}\ \{s', \ldots, s-1\}$: indeed, for any $y \in \{s', \ldots, s-1\}$ there exists $k_y\geq 1$ s.t. $y \in \{s', \ldots, s-1\}\cap (x+k_y t'+\delta_T \mathbb Z)$ and (\ref{libve}) readily implies $ x \ {\stackrel{\bf T}{\rightarrow}} \ y$. This holds for any $x \in \{s', \ldots, s-1\}$ and proves (\ref{zlibekjf}).

This implies $\{m_{\bf T}, \ldots, s-1\}\ {\stackrel{\bf T}{\rightarrow}} \{ m_{\bf T}, \ldots, s-1\}$ when $s' < t'$. It  remains to consider the case when $t'<s'$; we fix $x \in \{t', \ldots, s'-1\}$.

The same argument as above proves that  $ \{x\} \ {\stackrel{\bf T}{\rightarrow}} \  \{s', \ldots, s-1\}$.

Conversely, we decompose  $x$ as $x=t'+k_x s+r_x$ with $k_x \geq 0$ and $0\leq r_x<s$. For any $y \in \{s', \ldots, s-1\}$, property (\ref{zlibekjf}) yields $y\ {\stackrel{\bf T}{\rightarrow}} \  r_x$ since $0\leq r_x <s$; now, immediately, $r_x 
\ {\stackrel{\bf T}{\rightarrow}} \ t'+k_x s+r_x=x$ so that $y\ {\stackrel{\bf T}{\rightarrow}} \ x$ as expected.

\vspace{3mm} 

\noindent \underline{{\bf Step 3 -}{\it   Proof of  Theorem \ref{irreducibleclass} } }
  
 When $D$ and $D'$ are finite, we set $T=S_\mu\cup S_{\mu'}$   and apply Step 2.
 
 If $D'=-\infty$ and $D$ is finite, we set $T=T_{s'}= S_{\mu}\cup (S_{\mu'}\cap \{s', \ldots, -1\})$ with $s' \leq -1$, apply Step 2 then let $s'\to -\infty$. The two other cases are treated in the same way.
 
\end{proof} 

It is worth remarking that one may extend the above to adapt for $\mathcal{X}^{(\alpha)}$, that is to say, the chain (starting at any initial point) will be absorbed after finitely many steps by the essential class
\begin{align*}
 \left\{
\begin{array}{lll}
\{D',\ldots, D-1\}& {\rm if} & \alpha=0,\\
\{D',\ldots, D\}&{\rm if} & 0<\alpha<1,\\
\{D'+1, \ldots, D\}& {\rm if} & \alpha=1.
\end{array}
\right.
\end{align*}
 under the assumption of bounded jumps. The remaining statements of the theorem are proved in an analogous way.

\subsection{ The chain  $ \mathcal{X}^{(0)} $ in the  general case} 

We start by considering the irreducible classes of $ \mathcal X^{(0)}  $ on the additional assumptions that $S_\mu \cap \mathbb Z^- \neq \emptyset$ and $S_{\mu'} \cap \mathbb Z^+ \neq \emptyset$. Intuitively, the oscillating random walk  $\mathcal{X}^{(0)}$ starting from $0$ can visit arbitrarily large integers and so, it is quite natural to think that $ {\mathcal I}(0)$ contains the whole line $\mathbb Z$ in this case, under the hypothesis $d\wedge d'=1$. In fact, it is false and depends deeply on the structure of $S_\mu$ and $S_\mu'$. The following  theorem, which is based on the ideas developed in Step 2 above, clarifies this point.

\begin{theo}\label{generalcase}
We write $S_{\mu}^{+}$ (resp. $S_{\mu'}^{+}$) and $S_{\mu}^{-}$ (resp. $S_{\mu'}^{-}$) as the positive and negative components of $S_\mu$ (resp. $S_\mu'$), respectively. Suppose that these subsets are all non-empty. If $d \wedge d'=\delta$ and $d \neq d'$ then 

$\bullet$ \underline{Case when $D <d'$}\\
i) if $x \in \{D,\ldots,d'-1\}+d'\mathbb Z^{+}_0$ then $x$ is transient and its irreducible class is $$\mathcal{I}(x)=(x+d' \mathbb Z) \cap [D,+\infty[;$$  
ii) otherwise, $x$ is essential and its essential class is given by
$$\mathcal{I}(x)=(r_x+\delta \mathbb{Z}) \setminus (\{D,\ldots,d'-1\}+d' \mathbb Z^{+}_0).$$

$\bullet$ \underline{Case when $D'>-d$}

i) if $x \in \{-d,\ldots,D'-1\}+d\mathbb Z^{-}_0$ then $x$ is transient and its irreducible class is $$\mathcal{I}(x)=(x+d\mathbb{Z}) \,\cap \,]-\infty,D'-1];$$

ii) otherwise, $x$ is essential and its essential class is given by
$$\mathcal{I}(x)=(r_x+\delta \mathbb{Z}) \setminus (\{-d,\ldots,D'-1\}+d\mathbb Z^{-}_0).$$  
$\bullet$ \underline{Case when $D \geq d'$ and $D' \leq -d$} 

There is/are $\delta$ irreducible class(es), which is/are all essential:
$$r+\delta\mathbb Z \text{ with } 0\leq r <\delta.$$
\end{theo}

\begin{proof}
Before delving into details, we shall pay more attention to the fact that these two conditions cannot be attained at the same time due to $-D<d'<D'<d$ and we thus arrive at a contradiction. In particular, if $D<d'$ then $-D' \geq d$ and vice versa.
\newline
$\bullet$ \underline{Case when $D'<d$}  \\
i) The converse is easily done since $x\ {\stackrel{S_\mu \cup  S_{\mu'}}{\longleftrightarrow}}\ x+kd'$ for any $k \in \mathbb{Z}$ satisfying $x+kd' \geq D$. Indeed, the assumption of $S_{\mu'}$ leads to the semi-group generated by $S_{\mu'}$, says  $T_{\mu'}$, is equal to $d' \mathbb{Z}$. One can write $kd'=\displaystyle \sum s'_{i}$ as the finite sum of elements in $S_{\mu'}$. Selecting first the positive elements (if any) and then the negative ones, it immediately follows that $(x+d'\mathbb{Z}) \cap  [D,+\infty[\, \subset C(x)$.  
Now, we will show  by contraposition that $C(x) \subset (x+d'\mathbb{Z}) \cap  [D,+\infty[$.  Suppose $z \in ]-\infty,D-1] \cap C(x)$.
Let $\tau$ be the last time entering to the set $\{0,\ldots,D -1\}$ of $\mathcal{X}^{(0)}$ before visiting $x$ for the first time. Since $z\ {\stackrel{S_\mu \cup S_{\mu'}}{\longrightarrow}}\ x$, we have $\mathbb{P}_{z}[\tau < +\infty]=1$ and then put $X_\tau^{(0)}=y \in \{0,\ldots,D-1\}$. Observe that $x-y \in \{1,\ldots,d'-1\} +d'\mathbb Z^{+}_0$, i.e. $x-y \notin d'\mathbb{N}$ and thus, $z\ {\stackrel{S_\mu \cup S_{\mu'}}{\not\longrightarrow}}\ x$ (contradiction). When $z \geq D$, the crossing process starting at $z$ is directed only by $S_{\mu'}$ and therefore, $z \in x+ d' \mathbb Z$. \\
ii) A reasoning similar to the above yields that
 \begin{align*}
      \ y\ {\stackrel{S_{\mu} \cup S_{\mu'}}{\longleftrightarrow}} \ y+kd' \ \text{ \rm if } y\in \{0,\ldots,D-1\}+d' \mathbb Z^{+}_0, 
 \end{align*}
 where $k \in \mathbb Z \text{ s.t. } y+kd' \geq 0$.   \\
 Moreover, by Theorem \ref{irreducibleclass}
 \begin{align}\label{abc}
   \{D',\ldots,D-1\}\cap (r_x+\delta \mathbb{Z}) \ {\stackrel{S_{\mu}^{+} \cup S_{\mu'}^{-}}{\longrightarrow}} \{D',\ldots,D-1\}\cap (r_x+\delta \mathbb{Z}).  
 \end{align}
  
Now, we fixed $t \in ]-\infty,D'[ \,\cap \,(r_x+\delta \mathbb Z)$, then there is some $z_t \in \{D',\dotsi,-1\} \cap (r_x+\delta \mathbb Z)$ s.t. $z_t \equiv t $(mod $d$) due to $-D' \geq d$. Combining \eqref{abc} with the fact that $T_\mu=d \mathbb Z$, we get 
\begin{align*} 
   \{t\} \ {\stackrel{S_{\mu} \cup S_{\mu'}}{\longleftrightarrow}} \{D',\dotsi,D-1\}\cap (r_x+\delta \mathbb{Z}).  
 \end{align*}
 This property holds for every choice of $t$, so
 \begin{align*} 
   \{ ]-\infty,D'[ \,\cap \,(r_x+\delta \mathbb Z)\} \ {\stackrel{S_{\mu} \cup S_{\mu'}}{\longleftrightarrow}} \{D',\dotsi,D-1\}\cap (r_x+\delta \mathbb{Z}) 
 \end{align*}
 which achieves the proof.\\
 $\bullet$ We finish by mentioning that, suitably modified, the above argument applies
to the other cases.
\end{proof}

\begin{remark}{\label{ess_general} Suppose that $S_\mu$ is unbounded from above $(D=+\infty)$ and $S_{\mu'}$ is unbounded from below $(D'=-\infty)$. In such situation, all states connect and it is in fact possible that the process $\mathcal{X}^{(0)}$ with a positive probability will never reach a given neighborhood of $0$ due to infinitely many extremely large jumps across $0$. Consequently, there exists (with a positive probability) an orbit between every two points, which has no intermediary state belonging to a given finite set $F$, even if $0 \notin F$ (see \cite{Kemperman}). }
\end{remark}
\subsection{  On the invariant measure for $ \mathcal{X}^{(0)}$ when  $S_\mu \subset \mathbb Z^+$ and $S_{\mu'} \subset \mathbb Z^-$}

The concept of invariant measure plays a crucial role in the study in the long-time behaviour and asymptotic properties of a Markov chain.  To adapt for the current situation, a proper adjustment to the invariant measure in \cite{MijaVysot} is indispensable and what we found is the following 
\begin{lem}   { Assume $S_\mu \subset \mathbb Z^+$ and $S_{\mu'} \subset \mathbb Z^-$}. The measure $\lambda$ on $\mathbb{R}$ given by
$$\lambda(dx) = \left(\mathbb{1}_{]-\infty,0[}(x)\, \mathbb{P}[\xi_{1}' < x] + \mathbb{1}_{[0,+\infty[}(x)\, \mathbb{P}[\xi_{1} > x] \right)dx,$$
where $dx$ is Lebesgue measure, is invariant  for $\mathcal{X}^{(0)}$.
\end{lem}
\begin{proof}
 For any $a \geq 0$, it follows without difficulty that
\begin{align*}
    \lambda P ]a, +\infty[ &= \displaystyle \lim_{N \to +\infty} \displaystyle \int_{-N}^{0} \mathbb{P}[\xi'_{1} <x] \, \mathbb{P}[x+\xi_{1} >a] \, dx + \displaystyle \lim_{N \to +\infty} \displaystyle \int_{0}^{N} \mathbb{P}[\xi_1>x] \, \mathbb{P}[ x+\xi'_{1} > a] \, dx \\
     &=  \displaystyle \lim_{N \to +\infty} \displaystyle \int_{-N}^{0} \mathbb{P}[\xi'_{1} <x] \, \mathbb{P}[\xi_{1} >a-x] \, dx + \displaystyle \lim_{N \to +\infty} \displaystyle \int_{0}^{N} \mathbb{P}[\xi_1>x] \, \mathbb{P}[\xi'_{1} > a-x] \, dx \\
      &= \displaystyle \lim_{N \to +\infty} \displaystyle \int_{a}^{N+a} \mathbb{P}[\xi'_{1} <a-y] \, \mathbb{P}[\xi_{1} >y] \, dy + \displaystyle \lim_{N \to +\infty} \displaystyle \int_{a}^{N} \mathbb{P}[\xi_1>y] \, \mathbb{P}[\xi'_{1} > a-y] \, dy \\
      &= \displaystyle \lim_{N \to +\infty}\displaystyle \int_{a}^{N} \mathbb{P}[ \xi_{1} > y] \, dy+ \displaystyle \lim_{N \to +\infty}\displaystyle \int_{N}^{N+a}\mathbb{P}[\xi'_{1} <a-y] \, \mathbb{P}[\xi_{1} >y] \, dy\\
      &=\lambda ]a, +\infty[ 
\end{align*}
since $\displaystyle \int_{N}^{N+a}\mathbb{P}[\xi'_{1} <a-y] \, \mathbb{P}[\xi_{1} >y] \, dy \leq \displaystyle \int_{0}^{a} \mathbb{P}[\xi_1 > z+N]\,dz  \rightarrow 0$ as $N \rightarrow +\infty$.\\
The same computation yields $\lambda P ]-\infty, -a[ = \lambda ]-\infty, -a[ $ and thus, the proof is complete.
\end{proof}
As a direct consequence, we obtain the following result.
\begin{coro} \label{positive-recurrent}
Assume $S_\mu \subset \mathbb Z^+$ and $S_{\mu'} \subset \mathbb Z^-$. Then $\mathcal{X}^{(0)}$ is positive recurrent on each essential class iff both measures $\mu$ and $\mu'$ have finite first moment.
\end{coro}

\begin{lem}  Assume $S_\mu \subset \mathbb Z^+$ and $S_{\mu'} \subset \mathbb Z^-$. The discrete measure $\nu$ on $\mathbb{Z}$ given by
\begin{align}\label{discrete}
\nu(m):= \left\{
\begin{array}{lll}
   \mu']-\infty,m] &{\rm if} &m \leq -1\\ 
   \mu [m+1, +\infty[ &{\rm if}  &m\geq 0.  
\end{array}
\right.
\end{align}  
is invariant for the homogeneous random walk $\mathcal X^{(0)}$. Moreover, for arbitrary $x_0 \in \mathbb{Z}$, it induces a corresponding invariant measure $\nu_{x_0}$ on $\mathcal{I}(r_{x_0})$ by its restriction on this essential class (in other words, if $A \subset  \mathcal{I}(r_{x_0})$ then $\nu_{x_0}(A)= \displaystyle \sum_{x \in A}\nu(x)$).
\end{lem}
\begin{proof}
At first, we briefly outline the idea of the result. Let $X$ be a measurable space that is compatible with the $\text{Borel } \sigma-$algebra $\mathcal{B}$.  Suppose we have successfully founded an invariant measure $\lambda$ on  $(X,\mathcal{B})$ with the corresponding transition operator $P$. Let $\sim$ be an equivalence relation on $X$ and denote by $\overset{\sim}{X}:=X_{/_{\sim}}$ the quotient space of $X$ (whose equivalence classes belong to $\mathcal{B}$) under this relation. We assume $\overset{\sim}{X}$ is countable and holds for any elements $C_i, C_j \in \overset{\sim}{X}$, 
\begin{align}\label{quotientproperty}
    P(x,C_j)=P(y,C_j)  \text{ \rm for any }x,y \in C_i. \tag{*}
\end{align} 

Then the kernel $P$ induces a Markov transition $\overset{\sim}{P}$ on $\overset{\sim}{X}$ s.t. $\overset{\sim}{P}(C_i,C_j):=P(x,C_j)$ with $x \in C_i$; furthermore, the measure $\overset{\sim}{\lambda}$ on $\overset{\sim}{X}$ defined by $\overset{\sim}{\lambda}(C_i):=\lambda{(C_i)}$ is $\overset{\sim}{P}$- invariant. Indeed, for any $C' \in \overset{\sim}{X}$
\begin{align*}
    \overset{\sim}{\lambda}\overset{\sim}{P}(C')&= \displaystyle \sum_{C}\overset{\sim}{\lambda}(C)\overset{\sim}{P}(C,C')\\
    &=\displaystyle \sum_{C} \displaystyle \int_{C}P(x,C')\lambda(dx)\\
    &=\displaystyle \int_{X}P(x,C')\lambda(dx)\\
    &=\overset{\sim}{\lambda}(C'). 
\end{align*}

Let us now explain how to apply this general principle to get the exact formula of the invariant measure for the oscillating random walk. Consider the following equivalence relation 
$$x\sim y \Longleftrightarrow \exists n \in \mathbb{Z} \text{ s.t. }x,y \in [n,n+1[,$$
which apparently satisfies the condition \eqref{quotientproperty}. Taking $X=\mathbb R$ and $C_j:=[j,j+1[$ for any $j\in \mathbb Z$, one admits $\nu(m):=\lambda(C_m)$ (compare \eqref{discrete}) as the discrete invariant measure of $\mathcal{X}^{(0)}$.

\end{proof}
 
\begin{remark} When the state $0$ is supposed to be merged to the negative side $(\alpha=1)$, we replace $C_j$ by $C_j^*=]j-1,j]$ and the resulting invariant measure $\nu^*$ is given by
\[
\nu^*(m):= \left\{
\begin{array}{lll}
   \mu']-\infty,m-1] &{\rm if} &m \leq 0\\ 
   \mu [m, +\infty[ &{\rm if}  &m\geq 1.  
\end{array}
\right.
\]  
\end{remark}

\section{  On the crossing sub-process of $ \mathcal{X}^{(0)}$ when $S_\mu\subset \mathbb Z^+$ and $S_{\mu'}\subset \mathbb Z^-$}  
In this section, we would like to study the recurrence and the invariant measure of the embedding process of $\mathcal{X}^{(0)}$, which contains only states at the crossing times. For convenience, let us define the random variables $(S_n)_{n \geq 0}$ and $(S'_n)_{n \geq 0}$, the simple random walks associated with laws $\mu$ and $\mu'$ respectively by $S_0=S_0'=0$ and for $n\geq 1$,
   $$S_{n} = \xi_{1} + \xi_{2} + \dots + \xi_{n},$$
  and
   $$S'_{n} = \xi'_{1} + \xi'_{2} + \dots + \xi'_{n}.$$
  
   Denote by $\mu^{*n}$ the  $n$-fold convolution of $\mu$ with itself (also the distribution of $S_n$) and $\mathcal{U} = \displaystyle \sum_{n \geq 0}\mu^{*n}$ its potential kernel; similarly for $\mu'^{*n}$ and $\mathcal{U'}$. 
 Now we consider the sequence of crossing times ${\bf C}=(C_k)_{k \geq 0}$ at which the process changes its sign whenever crossing $0$. Assume that $C_0=0$ and we designate $C_k$ as the time of $k^{th}$- crossing given by 
\begin{align}\label{crossing_time}
    C_{k+1} := \inf \{n > C_k: X_{C_k}^{(0)} + (\xi_{C_k + 1} + \xi_{C_k + 2} + \dotsi +\xi_{n}) \geq 0 \text{ \rm if }  X_{C_k}^{(0)} \leq -1 \\ 
    \text{or } X_{C_k}^{(0)} + (\xi'_{C_k + 1} + \xi'_{C_k + 2} + \dotsi + \xi'_{n}) < 0 \text{ \rm if } X_{C_k}^{(0)} \geq 0 \}   \nonumber
\end{align}
This forms a sequence of stopping times with respect to the filtration $\mathbb{F}:=(F_n)_{n \geq 0}$ where $F_n :=\sigma \left(\xi_{k}, \xi'_{k}\mid k \leq n \right)$. By the law of large number, one gets $S_{n} \to +\infty$ and $S'_{n} \to -\infty$ $\mathbb{P}$-almost surely and thus, $\mathbb{P}_{x}[C_k <+\infty]=1$ for all $x \in \mathbb{Z}$.    
    
\begin{lem}  { Assume $S_\mu \subset \mathbb Z^+$ and $S_{\mu'} \subset \mathbb Z^-$.} The sub-process $(X_{C_k}^{(0)})_{k \geq 0}$ is a (time-homogeneous) Markov chain on $\mathbb{Z}$ with its transition kernel determined by 
  \begin{align}\label{matrixC} 
C(x,y)=\left\{
\begin{array}{lll}
\displaystyle \sum_{t = 0}^{-x - 1} \mu(y - x -t)\,\mathcal{U}(t) &\text{\rm if } x<0 \text{ \rm and } y\geq 0,\\   
\displaystyle \sum_{t= -x}^{0} \mu'(y - x - t) \, \mathcal{U'}(t) &\text{\rm if } x\geq 0 \text{ \rm and } y<0,\\
0 &{\rm otherwise}.
\end{array}
\right. 
 \end{align}    
 \end{lem} 
  The  process $\mathcal X^{(0)}_{\bf C}:= (X_{C_k}^{(0)})_{k \geq 0}$ is called the {\bf crossing sub-process}  of $\mathcal X^{(0)}$. 

 \begin{proof} 
 The Markov property is obvious.\\
 If $x < 0$ and $y \geq 0$ (similar to $x\geq 0,\, y<0$) then we have
\begin{align*}
    C(x, y) &= \mathbb{P} [X_{C_1} = y \mid X_{0} = x] \\
    &= \displaystyle \sum_{n=1}^{+\infty} \underbrace{\mathbb{P}[ x + S_{n-1} \leq -1,\, x + S_{n} = y]}_{\text{since }(S_n)_{n\geq 1} \text{ \rm is increasing}}\\
    &= \displaystyle \sum_{n=1}^{+\infty} \displaystyle \sum_{t = 0}^{-x-1} \mathbb{P}[S_{n-1} = t] \, \mathbb{P}[\xi_{n} = y - x - t]\\
    &= \displaystyle \sum_{t =0}^{-x-1} \mathbb{P}[\xi_{1} = y - x - t]\, \displaystyle \sum_{n=1}^{+\infty} \mathbb{P}[S_{n-1} = t]\\
    &= \displaystyle \sum_{t= 0}^{-x-1} \mu(y - x - t) \, \mathcal{U}(t).
\end{align*} 
\end{proof}   

\subsection{Irreducible classes  of $\mathcal X_{\bf C}^{(0)}$} 


 In case of  reflected random walk, it is well-known in \cite{MarcWoess} that the full process  and its process of reflections possess the common essential classes. 
Since the reflected random walk is regarded as the anti-symmetric case of our general model in which we identify the points themselves and their mirror images relative to $0$, it comes naturally a question whether this phenomenon possibly occurs. There is no solid information to give an exact answer other than the intuitive relationship $\mathcal{I}_{\bf C}(x) \subset \mathcal{I}(x)$, where $\mathcal{I}_{\bf C}(x)$ represents the irreducible class of $x$ with respect to the crossing sub-process $\mathcal X^{(0)}_{\bf C}$ starting at any given $x\in \mathbb Z$. Thus it is reasonable to attempt, using the below construction, to gain an understanding of the structure of $\mathcal{I}_{\bf C}$. 
 \begin{lem}  Assume $S_\mu \subset \mathbb Z^+$ and $S_{\mu'} \subset \mathbb Z^-$, and,  for any fixed $0 \leq r <\delta$, let us decompose  $\mathcal{I}(r)$ into $\mathcal{I}^{+}(r)\cup \mathcal{I}^{-}(r)$ where  $\mathcal{I}^{+}(r):= \mathcal{I}(r) \cap \mathbb Z_0^+$ and $\mathcal{I}^{-}(r):= \mathcal{I}(r) \cap \mathbb Z^-$. Set 
\begin{align}\label{plusorbit}
    \mathcal{I}_{\bf C}^+(r):= \{y \in \mathcal{I}^+(r): (y-S_{\mu})\cap \mathcal{I}^-(r) \neq \emptyset\},
\end{align}
and
\begin{align} \label{negativeorbit}
    \mathcal{I}_{\bf C} ^-(r):= \{y \in \mathcal{I}^-(r): (y-S_{\mu'})\cap \mathcal{I}^+(r) \neq \emptyset\}.
\end{align} 
Then $\mathcal{I}_{\bf C}(r)=\mathcal{I}_{\bf C}^+(r) \cup \mathcal{I}_{\bf C}^-(r)$ is an essential class of the crossing sub-process $\mathcal X_{\bf C}^{(0)}$. Furthermore, all the $X_{C_k}^{(0)}$ but finitely many belong to $\mathcal{I}_{\bf C}(r_{x_0})$ $\mathbb P$-a.s for any initial point $x_0\in \mathbb Z$.
 \end{lem}
 \begin{proof}
  For any $x,y \in \mathcal{I}_{\bf C}(r)$, we write $x \leadsto y$ to indicate that the crossing process $\mathcal{X}^{(0)}$ starting at $x$, reaches $y$ (with a positive probability) at certain crossing time $C_k$. Equivalently, there is such $z \in \mathcal{I}(r)$ and $n \in \mathbb{N}$ that $x \overset{n}{\rightarrow} z \overset{1}{\rightarrow} y$ \footnote{The notation means that $p^{(n)}(x,z)>0$ and $p(z,y)>0$.}, where $y$ and $z$ have the opposite signs.   
   
 The attractive property is immediate from the definition, so it remains to check that $\mathcal{I}_{\bf C}(r)$ is an irreducible class for $\mathcal{X}_{\bf C}^{(0)}$, i.e. $x \leadsto y$  for any given $x,y \in \mathcal{I}_{\bf C}(r)$. Without loss of generality, we suppose $x,y \in \mathcal{I}_{\bf C}^+(r)$. There exists some $s \in S_{\mu}$ and $n\geq 0$ s.t. $y-s \in \mathcal{I}^{-}(r)$ and $p^{(n)}(x,y-s)>0$. Then in a single step from $y-s$, the crossing process $\mathcal{X}^{(0)}$ reaches $y$ with the probability $\mu(s)>0$ at a crossing time as desired.
 \end{proof}
 
\begin{theo} \label{condition} Let $(a_i)_{i\geq 1}$ and $(b_j)_{j\geq 1}$ be strictly increasing sequences of positive integers. Set $S_\mu=(a_i)_{i\geq 1}$ and $S_{\mu'}=(-b_j)_{j\geq 1}$.  For any $0\leq r <\delta$, the structure of $\mathcal{I}_{\bf C}(r)$ and its complement will be completely revealed in view of the following conditions: 
\begin{enumerate}[(i).] 
\item  $D=+\infty \Longrightarrow \mathcal{I}_{\bf C}^-(r)=\mathcal{I}^-(r)$. 
\item  $D'=-\infty \Longrightarrow \mathcal{I}_{\bf C}^+(r)=\mathcal{I}^+(r)$.

 \item If $-D'<+\infty$ and $D=+\infty$ then
 $$\mathcal{I}_{\bf C}^+(r)=\mathcal{I}^+(r) \Longleftrightarrow \sup_{k \geq 1} \{a_k-a_{k-1}\} \leq -D' \text{ with } a_0=0.$$

 \item If $D'=-\infty$ and $D<+\infty$ then
 $$\mathcal{I}_{\bf C}^-(r)=\mathcal{I}^-(r) \Longleftrightarrow \sup_{\ell \geq 1}\{b_{\ell}-b_{\ell-1}\} \leq D \text{ with } b_0=0.$$  
 
\item If $D<+\infty$ and $-D'<+\infty$ then
$$\mathcal{I}_{\bf C}^+(r)=\mathcal{I}^+(r) \Longleftrightarrow \underset{1\leq k \leq m}{\max}\{a_{k}-a_{k-1}\} \leq -D',  $$ 
and 
$$\mathcal{I}_{\bf C}^-(r)=\mathcal{I}^-(r) \Longleftrightarrow \underset{1\leq \ell \leq n}{\max}\{b_{\ell}-b_{\ell-1}\} \leq D, $$
where $a_0=b_0=0; D=a_m$ and $D'=-b_n$ for some $m,n\geq 1$.

\item Set $I^+:=\{k\geq 1 \mid a_k-a_{k-1} > -D'\}$ and $I^-:=\{\ell \geq  1 \mid b_{\ell}-b_{\ell-1}>D\}$ in case (v) is violated. For every choice of $k \in I^+$ and $\ell \in I^-$,  we define
$$\mathcal{N}_k^+(r):=\left\{a_{k-1}+r+\delta s: 0\leq s\leq \dfrac{a_{k}-a_{k-1}+D'}{\delta}-1\right\},$$
and
$$\mathcal{N}_\ell^-(r):=\left\{-b_{\ell-1}+r+\delta s: \dfrac{b_{\ell-1}-b_\ell+D}{\delta}\leq s \leq -1\right\}.$$
Then $\mathcal{N}(r):= \underset{(k, \ell) \in I^+ \times I^-}{\bigcup} \mathcal{N}_k^+(r) \cup \mathcal{N}_\ell^-(r)$ is  the set of \text{non-crossing} points. 
\end{enumerate}
 \end{theo}
 
\begin{proof} (i)-(ii). By definition.\\
(iii)-(iv). Note that \eqref{plusorbit}  and \eqref{negativeorbit} can be rewritten as
\[
\mathcal{I}_{\bf C}^+(r)= \mathcal{I}^+(r) \cap \bigcup_{k\geq 1} A_k \text{ \rm with } A_k:=\{a_k+D'+r,\ldots,a_k+r-\delta\}
\]
and
\[
\mathcal{I}_{\bf C}^-(r)= \mathcal{I}^-(r) \cap \bigcup_{\ell\geq 1}B_{\ell} \text{ \rm with } B_{\ell}:=\{-b_{\ell}+r,\ldots,-b_{\ell}+ D+r-\delta\}.
\]
 
To cover $\mathcal{I}^+(r)$ by countably many same length sub-intervals, it requires $r \in A_1$ and further, no point in the form $r+\delta \mathbb Z$ stays inside the gap between $A_k$ and $A_{k+1}$ since $\{a_k+D'+r\}_{k\ge 1}$ is a strictly increasing sequence. More precisely, 
\begin{align*} 
\mathcal{I}_{\bf C}^+(r)=\mathcal{I}^+(r) &\Longleftrightarrow \left\{
\begin{array}{lll}
a_1+D'+r \leq r    \\
a_{k+1}+D'+r\leq a_k+r, \forall k\geq 1 \\ 
\end{array}
\right.    
&\Longleftrightarrow \left\{
\begin{array}{lll}
a_1\leq -D'  \\
a_{k+1}-a_k\leq -D',  \forall k\geq 1 \\ 
\end{array}
\right.    
 \end{align*}  
Identically, we also infer
  \begin{align*} 
\mathcal{I}_{\bf C}^-(r)=\mathcal{I}^-(r) &\Longleftrightarrow \left\{
\begin{array}{lll}
-b_1+D+r-\delta \geq r-\delta    \\
-b_\ell+r \leq -b_{\ell+1}+D+r, \forall \ell \geq 1 \\ 
\end{array}
\right.    
&\Longleftrightarrow \left\{
\begin{array}{lll}
b_1\leq D  \\
b_{\ell+1}-b_\ell \leq D,  \forall \ell\geq 1. \\ 
\end{array}
\right.    
 \end{align*}  
(v). This is a direct consequence of (iii) and (iv).\\
(vi). A straightforward argument yields
$$(\mathcal{N}_k^+(r)-S_\mu) \cap \mathcal{I}^-(r)= (\mathcal{N}_\ell^-(r)-S_{\mu'}) \cap \mathcal{I}^+(r)=\emptyset$$
 for every pair $(k,\ell) \in I^+\times I^-$ and it is enough to end the proof. 
\end{proof}

 \begin{exam} We simply deal with the case when $D, -D'<+\infty$ by, for instance, taking $S_\mu:=\{2, 4, 10\}$ and $S_\mu':=\{-4, -1\}$. Obviously, $D=10, D'=-4, d=2, d'=1$ and the chain $\mathcal{X}^{(0)}$ has the unique essential class $\mathcal{I}(0):=\{-4,\ldots, 9\}$. An easy verification on (v) would lead to $\mathcal{I}_{\bf C}^-(0)=\mathcal{I}^-(0)$ while $\mathcal{I}_{\bf C}^+(0) \varsubsetneq \mathcal{I}^+(0)$ and its complement $\mathcal{N}_3^{+}=\{4, 5\}$ according to (vi). Hence, equality is achieved only by replacing $S_\mu$ by $\{2, 6, 10\}$. 
 \end{exam}
 
\begin{remark} We do emphasize that it fails to let $\mathcal{I}_{\bf C}^+(r) \varsubsetneq \mathcal{I}^+(r)$ and $\mathcal{I}_{\bf C}^-(r) \varsubsetneq \mathcal{I}^-(r)$ simultaneously occur due to (v). In other words, $\mathcal{I}_{\bf C}(r)$ always has at least one side which coincides with $\mathcal{I}(r)$. The case when $S_{\mu} \subset \mathbb Z$ and $S_{\mu'} \subset \mathbb Z$ is much complicated since the behaviour of the chain now is  significantly affected by many factors. However, in connection with Theorem \ref{generalcase} and the above theorem, one may derive some properties of $\mathcal{I}_{\bf C}$, for instance, $\mathcal{I}_{\bf C}(r)=\mathcal{I}(r)$ if and only if $D=-D'=+\infty$.   
 \end{remark}
 
 \subsection{Invariant measure for $\mathcal X_{\bf C}^{(0)}$}  
 
Use of the explicit formula \eqref{discrete} of $\nu$ enables us to derive the invariant measure of the  sub-process $\mathcal X_{\bf C}^{(0)}$. In particular 
\begin{theo} 
Assume $S_\mu \subset \mathbb Z^+$ and $S_{\mu'} \subset \mathbb Z^-$. Let $\rho$ be the measure on $\mathbb Z$ defined by 
\begin{align}\label{formula_rho}
\rho(n):=\left\{
\begin{array}{lll}
\displaystyle \sum_{k=1}^{+\infty} \mu(k)\,\mu'[n-k+1,n] &{\rm if} &n \leq -1, \\
\displaystyle \sum_{k=1}^{+\infty} \mu'(-k)\,\mu[n+1,n+k] &{\rm if} &n \geq 0.
\end{array} 
\right. 
\end{align}

Then,  for any $x_0\in \mathbb Z$, the restriction $\rho_{x_0}$ of $\rho$ to $\mathcal{I}_{\bf C}(r_{x_0})$ 
is an invariant measure for $\mathcal X_{\bf C}^{(0)}$.
 \end{theo}
\begin{proof}  

Consider the signed measures $\mathcal{A}$ and $\mathcal{A}'$ defined by $\mathcal{A}(m) = \delta_{0}(m) - \mu(m)$ if $m \geq 0$ and $\mathcal{A}'(m) = \delta_{0}(m) - \mu'(m)$ if $m \leq 0$. It is easily seen that
\begin{align} \label{convolution}
    \mathcal{A}* \mathcal{U} = \mathcal{U}*\mathcal{A} = \delta_{0} \text{ and } \mathcal{A}'*\mathcal{U}' = \mathcal{U}'* \mathcal{A}' = \delta_{0}. 
\end{align} 
In addition, we also have
\begin{align} \label{rho-nu}
   \rho(n) := \left\{
   \begin{array}{lll}
   \displaystyle \sum_{k = 0}^{+\infty}\mathcal{A}(k)\,\nu(n-k) &{\rm if} &n \leq -1,\\
 \displaystyle \sum_{k = 0}^{+\infty}\mathcal{A}'(-k)\,\nu(n+k) &{\rm if} &n \geq 0.
   \end{array}
    \right.
\end{align}  
Indeed, if $n \geq 0$ then we get 
\begin{align*}
    \displaystyle \sum_{k = 0}^{+\infty}\mathcal{A}'(-k)\,\nu(n+k) &= \displaystyle \sum_{k=0}^{+\infty} \left(\delta_{0}(-k) - \mu'(-k)\right)\, \nu(n + k)\\
    &= (1 - \mu'(0))\,\nu(n) - \displaystyle \sum_{k=1}^{+\infty} \mu'(-k)\nu(n + k)\\ 
    &= \displaystyle \sum_{k=1}^{+\infty} \mu'(-k)\,\left(\mu[n+1, +\infty[ - \mu[n+k+1, +\infty[ \right)\\
    &=\displaystyle \sum_{k=1}^{+\infty} \mu'(-k)\,\mu[n+1,n+k].
\end{align*} 

Conversely, $\nu$ can be represented in terms of $\rho$ and $\mathcal{U}'$ directly from \eqref{convolution} and \eqref{rho-nu}, that is, for $n \geq 0$ 
\begin{align*}
    \displaystyle \sum_{k = 0}^{+\infty}\mathcal{U}'(-k)\,\rho(n + k) &=  \displaystyle \sum_{k = 0}^{+\infty}\mathcal{U}'(-k)\,\displaystyle \sum_{\ell = 0}^{+\infty}\mathcal{A}'(-\ell)\,\nu(n + k + \ell) \\
    &=\displaystyle \sum_{k=0}^{+\infty}\mathcal{U}'(-k)\,\displaystyle \sum_{s=k}^{+\infty}\mathcal{A}'(k-s)\nu(n+s)\\
    &=  \displaystyle \sum_{s = 0}^{+\infty} \nu(n+s)\, \underbrace{\displaystyle \sum_{k=0}^{s}\mathcal{U}'(-k)\,\mathcal{A}'(-s+k)}_{\delta_0(-s)} \\
    &= \nu(n)
\end{align*} 
and the same property holds for $n \leq -1$. Briefly,  one may write 
\begin{align} \label{nu-rho}
   \nu(n) :=  \mathbb{1}_{]-\infty, 0[}(n)\left(\displaystyle \sum_{k = 0}^{+\infty}\mathcal{U}(k)\,\rho(n - k)\right)+\mathbb{1}_{[0, +\infty[}(n)\left(\displaystyle \sum_{k = 0}^{+\infty}\mathcal{U}'(-k)\,\rho(n + k)\right).
\end{align}   
We claim that
\begin{align} \label{mainformula}
    \nu_{x_0}(n) := \mathbb{E}_{\rho_{x_{0}}}\left(\displaystyle \sum_{j = 0}^{C_1 - 1} \mathbb{1}_{n}(X_j^{(0)}) \right),\,\,\,\,\,\,\,\,\, \text{if } n \in \mathcal{I}(r_{x_0})   
\end{align} 
where $\mathbb{E}_{\rho_{x_0}}(.) = \displaystyle \sum_{w \in \mathcal{I}_{\bf C}(r_0)} \rho_{x_0}(w)\,\mathbb{E}_{w}(.)$ indicates the expectation governed by $\rho_{x_0}$.\\

If $n \in \mathcal{I}^-(r_{x_0})$ then the crossing process can reach $n$ before the first crossing time if and only if $X_0^{(0)}=\omega \in \mathcal{I}_{\bf C}^-(r_{x_0})$ and $\omega \leq n$. In other words, there is $k \in \mathbb{Z}_0^+$ and $i \geq 1$ s.t. $S_i=k$ and $\omega=n-k$. Since $\displaystyle \sum_{j = 0}^{C_1 - 1} \mathbb{1}_{n}(X_j^{(0)}) = \displaystyle \sum_{j=0}^{+\infty} \mathbb{1}_{n}(X_j^{(0)})\,\mathbb{1}_{\{C_{1}>j\}}$,  
\begin{align*}
    \mathbb{E}_{\rho_{x_0}} \left(\displaystyle \sum_{j = 0}^{C_1 - 1} \mathbb{1}_{n}(X_j^{(0)}) \right) &= \displaystyle \sum_{k=0}^{+\infty} \rho_{x_0}(n - k)\, \mathbb{E}_{n - k} \left(\displaystyle \sum_{j=0}^{+\infty} \mathbb{1}_{n}(X_j^{(0)})\,\mathbb{1}_{\{C_{1}>j\}}\right) \\
    &= \displaystyle \sum_{k=0}^{+\infty}\rho_{x_0}(n-k) \displaystyle \sum_{j=0}^{+\infty} \mathbb{P}_{n-k}[X_j^{(0)}=n, C_1 > j]\\
    &=\displaystyle \sum_{k=0}^{+\infty}\rho_{x_0}(n-k) \displaystyle \sum_{j=0}^{+\infty} \mathbb{P}[S_j=k]\\ 
    &= \displaystyle \sum_{k=0}^{+\infty}\rho_{x_0}(n - k)\, \mathcal{U}(k) \\
   &= \nu_{x_0}(n).
\end{align*}
Hence, \eqref{mainformula} is true for every $n \in \mathcal{I}(r_{x_0})$ and yields that
$$\displaystyle \sum_{m \in \mathcal{I}(r_{x_0})} \nu_{x_0}(m)\,p(m, n) = \mathbb{E}_{\rho_{x_0}}\left(\displaystyle \sum_{j = 1}^{C_1} \mathbb{1}_{n}(X_j^{(0)}) \right).$$
Since $\nu_{x_0}$ is invariant on $\mathcal{I}(r_{x_0})$, i.e. $\nu_{x_0} = \nu_{x_0} \, P$, we again apply  \eqref{mainformula} and simplify as 
$$\mathbb{E}_{\rho_{x_0}} \left(\mathbb{1}_{n}(X_0^{(0)})\right) = \mathbb{E}_{\rho_{x_0}}\left(\mathbb{1}_{n}(X_{C_1}^{(0)})\right).$$
The left hand side is $\rho_{x_0}(n)$ and the right hand side is the sum $\displaystyle \sum_{m \in \mathcal{I}_{\bf C}(r_{x_0})} \rho_{x_0}(m)\, C(m, n)$, which prove that $\rho_{x_0}$ is an invariant measure for $  \mathcal X_{\bf C}^{(0)}$ on $\mathcal{I}_{\bf C}(r_{x_0})$. 
\end{proof}

\section{Criteria for the recurrence of $ \mathcal{X}^{(\alpha)}$,  $  0\leq \alpha \leq 1$}  
{\bf From now on, we assume $d=d'=1$} and keep in mind that the recurrence of $\mathcal X^{(\alpha)} $ always means the recurrence of the state $0$ since  {\bf $\mathcal X^{(\alpha)} $  has an unique  irreducible class in this case}.   

\subsection{Classical approach}
In this subsection, we consider the first passage of $(S_n)_{n\geq 0}$ to the subset $]0, +\infty[$ and of $(S_n')_{n\geq 0}$ to the subset $]-\infty, 0[$, namely
\begin{align*}
     \ell_+:=\inf\{k>0: S_k>0\} \text{ and } \ell_-':=\inf\{k>0: S_k'<0\} 
\end{align*}
(with the convention $\inf \emptyset=+\infty$). The random variables $\ell^+$ and $\ell_-'$ are stopping times with respect to the canonical filtration $(\sigma(\xi_k, \xi_k', 1\leq k \leq n))_{n\geq 1}$ In the sequel, we
only consider cases when these random variables are $P$-as finite i.e. equivalently when
$\mathbb{P}[\limsup S_n=+\infty]= \mathbb{P}[\liminf S_n'=-\infty] =1$; hence, the random variables $S_{\ell_+}$ and $S_{\ell_-'}'$
are well defined in these cases and \textbf{we denote by $ \mu_+$ and $ \mu'_-$ their respective distributions}

Let us define also, for $h\geq 1$, the renewal functions associated with  the ladder heights $S_{\ell^+}$ and $S_{\ell^-}$, respectively by
\begin{align*}
   & C(h):= \displaystyle \sum_{n=1}^{+\infty} \mathbb P[S_n=h, \min_{1\leq i\leq n}S_i >0],\\
   & C'(-h):= \displaystyle \sum_{n=1}^{+\infty} \mathbb P[S_n'=-h, \min_{1\leq j\leq n}S_j' >0]. 
\end{align*} 
We now get a glimpse of the following well-known criterion of recurrence of $ \mathcal{X}^{(\alpha)}$ 
\begin{theo} \text{(Kemperman)} The general oscillating random walk $\mathcal{X}^{(\alpha)}$ is recurrent if and only if
 \begin{align}\label{Kemperman_cond}
     \displaystyle \sum_{h=1}^{+\infty} C(h)\,C'(-h)=+\infty.
 \end{align}
 \end{theo}
However, it is quite theoretical and difficult to check in several cases. Next, we will take into consideration an equivalent condition to \eqref{Kemperman_cond}, which has the additional advantage of being easily computable.

\begin{theo} (\text{Rogozin-Foss}) If for some $\epsilon>0$
\begin{align}
    \displaystyle \int_{-\epsilon}^{\epsilon} \dfrac{1}{\vert 1-\mathbb E[e^{itS_{\ell_+}}]\vert \, \vert 1-\mathbb E[e^{itS_{\ell_-'}'}]\vert }\, dt <+\infty,
\end{align}
then the original crossing process is transient.\\
If, in addition,  $\text{Re }((1-\mathbb E[e^{itS_{\ell_+}}])(1-\mathbb E[e^{itS_{\ell_-'}'}])) \geq 0$ for $\vert t \vert <\epsilon$ for some $\epsilon>0$  and the below condition holds 
\begin{align}
   \displaystyle \int_{-\epsilon}^{\epsilon} \text{Re }\left(\dfrac{1}{(1-\mathbb E[e^{itS_{\ell_+}}]) \, (1-\mathbb E[e^{itS_{\ell_-'}'}])} \right) \,dt =+\infty,
\end{align}
then $\mathcal{X}^{(\alpha)}$ is recurrent.
 \end{theo}
 We also refer to the recent paper \cite{Bremont}, Proposition $4.4$ for a luminous proof. 
  
In the simple case when $\mu=\mu$, in other words when $\mathcal{X}^{(\alpha)}$ is an homogeneous classical
random walk on $\mathbb Z$, these conditions turn into the above conditions turn into
\begin{align}
    \displaystyle \int_{\epsilon}^{\epsilon} \vert 1-\hat{\mu}(t) \vert^{-1} dt < +\infty,
\end{align}
and
\begin{align}\label{Kesten}
    \displaystyle \int_{\epsilon}^{\epsilon} Re\left(\dfrac{1}{1-\hat{\mu}(t)}\right)  dt = +\infty,
\end{align}
 In fact, \eqref{Kesten} is a necessary and sufficient
condition for $(S_n)_{n\geq 0}$ be recurrent; see \textsc{Kesten}, \textsc{Spitzer} \cite{KestenSpitzer}, \cite{Revuz} and \cite{Bremont} for proofs and comments. In \cite{Bremont}, Proposition $2.2$, the reader will find an explicit and simple relation between the integral of the function $\text{Re}\left(\dfrac{1}{1-\hat{\mu}(t)}\right)$ and the Green function of the random walk $(S_n)_n$ which enlightens the above statement.
     
 In the next subsection we develop another approach to identify quite general conditions which ensure that   $\mathcal X^{(\alpha)}$ is recurrent. We first consider the case when $S_\mu\subset \mathbb Z^+$ and $S_{\mu'}\subset \mathbb Z^-$ then the general case, replacing the couple $(\mu, \mu')$ by $( \mu_+,  \mu'_-)$.
\subsection{ Tail condition criterion for  the recurrence of $ \mathcal{X}^{(0)}$ when $S_\mu\subset \mathbb Z^+$ and $S_{\mu'}\subset \mathbb Z^-$} 
 An easy observation gives that the crossing sub-process $  \mathcal X_{\bf C}^{(0)}$ is positive recurrent on $\mathcal{I}_{\bf C}(0)$ $(\text{equivalently, }\rho(\mathcal{I}_{\bf C}(0))<+\infty)$  when $D$ and  $D'$  are both finite.  Thus, it is reasonable to study the recurrence of $  \mathcal X_{\bf C}^{(0)}$ in the non-trivial cases.

\begin{prop} \label{tail} Assume that $S_\mu\subset \mathbb Z^+$ and $S_{\mu'}\subset \mathbb Z^-$. Then the total mass of $\rho$ on $\mathcal{I}_{\bf C}(0)$  is finite if and only if
\begin{align}\label{totalmass}
    \displaystyle \sum_{n = 0}^{+\infty} H(n)\,H'(-n) < +\infty,
\end{align}
where $H(n) = \mu]n,+\infty[$ and $H'(-n) = \mu']-\infty, -n[$ respectively stands for the tail distributions of $\mu$ and $\mu'$.

In this case, the Markov chain $\mathcal X_{\bf C}^{(0)}$ is positive recurrent and $\mathcal X^{(0)}(\mu, \mu')$ is recurrent on its essential class.  
\end{prop}
\begin{proof}
  We  compute $\rho(\mathbb Z^+)$ by substituting the formula \eqref{formula_rho}
 \begin{align*}
     \displaystyle \sum_{n=0}^{+\infty}\rho(n) &= \displaystyle \sum_{n=0}^{+\infty}\displaystyle \sum_{k=1}^{+\infty}\mu'(-k)\,\mu[n+1,n+k]\\ 
     &=\displaystyle \sum_{k=1}^{+\infty}\mu'(-k)\displaystyle \sum_{n=0}^{+\infty}[H(n)-H(n+k)] \\  
      &= \displaystyle \sum_{k=1}^{+\infty}\mu'(-k) \left[\displaystyle \sum_{n=0}^{k-1}H(n)- \displaystyle \lim_{N \to +\infty}\left(H(N+1)+ \ldots +H(N+k)\right)\right] \\ 
      &= \displaystyle \sum_{n=0}^{+\infty}H(n) \displaystyle \sum_{k=n+1}^{+\infty}\mu'(-k)\\
      &= \displaystyle \sum_{n= 0}^{+\infty}H(n)\,H'(-n).
 \end{align*}

 One also obtains $\displaystyle \sum_{n=-\infty}^{-1}\rho(n)=\displaystyle \sum_{n=0}^{+\infty}H(n)\,H'(-n)$ which immediately implies  \eqref{totalmass}. 
\end{proof}

Apparently, \eqref{totalmass} holds providing that the first moment of either $\xi_n$ or $-\xi_n'$ is finite. This assumption can be sharpen by constraining finite H{\"o}lder moments as  below

\begin{coro} \label{Holder} Assume that $S_\mu\subset \mathbb Z^+, S_{\mu'}\subset \mathbb Z^-$ and $\mathbb{E}[\xi_1^{p}], \mathbb{E}[(-\xi_1')^q] <+\infty$ with $p, q \in ]0,1[$ satisfying $p+ q=1$.  Then \eqref{totalmass} holds and the Markov chain  $\mathcal X^{(0)}$ is recurrent on its unique essential class (and positive recurrent when $\mathbb{E}[\xi_1], \mathbb{E}[-\xi_1'] < +\infty$ by  Corollary \ref{positive-recurrent}).
\end{coro}
\begin{proof}
The formula $\mathbb{E}[X^k]:=k \displaystyle \int_{0}^{+\infty}t^{k-1} \mathbb{P}[X\geq t]\,dt$ yields 
\begin{align}\label{moment}
     k\displaystyle \sum_{n=0}^{+\infty}(n+1)^{k-1}\mathbb{P}[X\geq n+1] \leq \mathbb{E}[X^k]
 \end{align}
so that, by the \textit{Markov's inequality} for $H^{p}(n)$ and $H'^{q}(n)$, we get  
 \begin{align*}
    \displaystyle \sum_{n=0}^{+\infty} H(n)\, H'(-n)&= \displaystyle \sum_{n=0}^{+\infty}H^{q}(n)H'^{p}(-n)\left[H^{p}(n)H'^{q}(-n)\right] \\
    &\leq \mathbb{E}\left[\xi_1^p\right]^p\,\mathbb{E} [(-{\xi_1'})^q]^q \displaystyle \sum_{n=0}^{+\infty}  \left[(n+1)^{-q^2}H^{q}(n) \, (n+1)^{-p^2}H'^{p}(-n) \right]. 
\end{align*} 
 The product inside the bracket can be transformed into sum by using the \textit{Young's inequality} and then together with \eqref{moment}, it yields
 \begin{align*}
     \displaystyle \sum_{n=0}^{+\infty} H(n)\, H'(-n) &\leq \mathbb{E}\left[\xi_1^p\right]^p\,\mathbb{E} [(-\xi_1')^q]^q \left(q \displaystyle \sum_{n=0}^{+\infty} (n+1)^{-q}H(n)+p \displaystyle \sum_{n=0}^{+\infty} (n+1)^{-p}H'(-n) \right)\\
     &\leq \mathbb{E}\left[\xi_1^p\right]^p\,\mathbb{E} [(-\xi_1')^q]^q \left(\dfrac{q}{p} \mathbb{E}\left[\xi_1^p\right]+\dfrac{p}{q} \mathbb{E} [(-\xi_1')^q] \right)<+\infty.
 \end{align*} 
 \end{proof}
\begin{remark}
 The condition $\mathbb{E}[\xi_1^{p}], \mathbb{E}[(-\xi_1')^q] <+\infty$ with $p+ q=1$ is a sufficient condition for the recurrence of $\mathcal{X}^{(0)}$. Notice that it is not far to be sharp. We refer to Proposition  $5.12$ in \cite{MarcWoess} for an example in the case of the reflected random walk on $\mathbb N$, which corresponds to the antisymmetric case, i.e. $S=-S'$ (with $p=q= 1/2$ there). The reader can find other examples in \cite{RogozinFoss} Theorem $2$ in the case when $S$ and $S'$ are stable random walks on $\mathbb Z$ but $S \neq -S'$.

\end{remark}

\subsection{Recurrence of $ \mathcal{X}^{0)}$ in the general case} 
To treat this model, let us first introduce the basic decomposition of $\xi_n$ and $\xi_n'$, namely
$$\xi_n=\xi_n^+-\xi_n^- \quad \text{and } \quad  \xi_n'=\xi_n'^+-\xi_n'^-,$$
where $\xi_n^{\pm}=\max\{\pm \xi_n,0\}$ and $\xi_n'^{\pm}=\max\{\pm \xi_n',0\}$.

Consider the following assumptions\\
$\bf{(H)} \,\, \left(\mathbb E[\xi_1^-]<\mathbb E[\xi_1^+] \leq +\infty\right)$ or $\left(\mathbb E[\xi_1^-]=\mathbb E[\xi_1^+]<+\infty\right);$\\
$\bf{(H')} \,\, \left(\mathbb E[(\xi_1')^+]<\mathbb E[(\xi_1')^-] \leq +\infty\right)$ or $\left(\mathbb E[(\xi_1')^-]=\mathbb E[(\xi_1')^+]<+\infty\right).$
\begin{lem} \label{abcd} If $\bf H\, (\text{resp. } \bf{H'})$ is satisfied then $\limsup\limits S_n=+\infty\, (\text{resp. } \liminf\limits S_n'=-\infty)$ almost surely. 
\end{lem}
Hence, when  both $\bf H$ and $\bf{H'}$ hold, the random variables $\ell_+$ and $\ell_-'$ are $\mathbb P$-a.s. finite; more generally, there are infinitely many  $\mathbb P$-a.s. finite  crossing times in this case. Let us introduce the ladder times $(t_k)_{k\geq 0}$ defined recursively by: $t_0=0$ and, for $k\geq 1$, 
\begin{align}
t_{k+1}:=\left\{
\begin{array}{lll}
\inf\{n>t_k \mid \xi_{t_k+1}+\ldots+\xi_n>0\}& {\rm if} & X_{t_k}^{(0)}\leq -1, \\ \\
\inf\{n>t_k \mid \xi_{t_k+1}'+\ldots+\xi_n'<0\}&{\rm if} & X_{t_k}^{(0)} \geq 0. 
\end{array}
\right.
\end{align}

Notice that, in the first line, the random variable $t_{k+1}$ is an ascending ladder epoch of $S$, while, in the second line,  it is a descending ladder epoch of $S'$. These random times are $\mathbb P$-a.s. finite and the increments $(t_{k+1}-t_k)_{k\geq 0}$ form  a sequence of independent random variables with laws  
\begin{align*}
 \mathcal{L}(t_{k+1}-t_k \mid X_{t_k}^{(0)}<0) =\mathcal{L}(t_1 \mid X_0^{(0)}=x) \quad \text{and} \quad
  \mathcal{L}(t_{k+1}-t_k \mid X_{t_k}^{(0)}\geq 0) =\mathcal{L}(t_1 \mid X_0^{(0)}=y)  
\end{align*}
for any $x < 0 \leq y$. Similarly,  
$$ \mathcal{L}(S_{t_{k+1}}-S_{t_k} \mid X_{t_k}^{(0)}<0) =\mathcal{L}(S_{t_1} \mid X_0^{(0)}=x)= \mu_+, $$
  and
 $$ \mathcal{L}(S_{t_{k+1}}'-S_{t_k}' \mid X_{t_k}^{(0)}\geq 0) =\mathcal{L}(S_{t_1}' \mid X_0^{(0)}=y)= \mu'_-. 
$$ 
It is perhaps worth remarking that, by setting $Y_k:=S_{t_k}-S_{t_{k-1}}$ when $X_{t_k}^{(0)} < 0$ and $Y_k':=S_{t_k}'-S_{t_{k-1}}'$ when $X_{t_k}^{(0)} \geq 0$, the sub-process $(X_{t_k}^{(0)})_{k\geq 0}$ turns out to be an crossing process associated with the distributions $ \mu_+$ and $ \mu'_-$ of $Y_k$ and $Y_k'$ respectively. In other words, the process $(X_{t_k}^{(0)})_{k\geq 0}$ has the same distribution as $\mathcal{X}^{(0)}( \mu_+,  \mu'_-)$.  
 \begin{lem}\label{lem} Assume that both $\bf H$ and $\bf{H'}$ hold. Then, the oscillating process $ \mathcal X^{(0)}(\mu, \mu')$ is recurrent if and only if the oscillating process $\mathcal{X}^{(0)}( \mu_+,  \mu'_-)$ is recurrent.
 \end{lem}
 A proof of this statement for the process $\mathcal{X}^{(1)}$ appears in the recent paper \cite{Bremont} of \textsc{J. Bremont}, lemma $4.2 \, (ii)$. For the sake of completeness, we detail the argument below,
introducing the first return time at $0$ of $\mathcal{X}^{(\alpha)}, 0 \leq \alpha \leq 1$, which will be useful latter on.
 \begin{proof} 
  For any $0\leq \alpha \leq 1$, let $\uptau^{(\alpha)}$  be the first return time at $0$ of $\mathcal{X}^{(\alpha)}$ given by
 $$\uptau^{(\alpha)}:=\inf\{n\geq 1: X_n^{(\alpha)}=0\}.$$
 In the present proof, we only consider the case when $\alpha = 0$.
 
 Starting at $0$, we know that $\uptau^{(0)}<+\infty$ almost surely and since $X_{\uptau^{(0)}}^{(0)}=0$, there are only two possibilities: if $X_{\uptau^{(0)}-1}^{(0)} \geq 1$ then $\uptau^{(0)}$ must be a ladder time while the case $X_{\uptau^{(0)}-1}^{(0)} \leq -1$ will imply the existence of some $k\geq 1$ s.t. $\uptau^{(0)}=C_k$, the $k^{th}$-crossing time of $\mathcal{X}^{(0)}$ and of course that $\uptau^{(0)}$ is also a ladder time. This means $(X_{t_k}^{(0)})_{k\geq 0}$ is recurrent ans so does $\mathcal{X}^{(0)}( \mu_+,  \mu'_-)$. The converse is obvious.
 \end{proof}
 
  It is easily seen that $\mathcal{X}^{(0)}$ and $(X_{t_k}^{(0)})_{k\geq 0}$ admit a common crossing sub-process  since there is at most a crossing moment between two consecutive ladder times $t_k$ and $t_{k+1}$ happening when $X_{t_k}^{(0)} <0$ and $X_{t_{k+1}}^{(0)}\geq 0$ or vice versa. Therefore, we can take advantage of Corollary \ref{Holder} (applied to the process $\mathcal{X}^{(0)}( \mu_+,  \mu'_-)$) to deduce the recurrence of $(X_{t_k}^{(0)})_{k\geq 0}$ and finally that of $\mathcal{X}^{(0)}$ by Lemma \ref{lem}. 
  
 \begin{theo} \label{extend_condition} Let $p, q\in ]0,1[$ s.t. $p+q=1$. Then each of the following is sufficient for the oscillating process $\mathcal{X}^{(0)}(\mu, \mu')$ to be recurrent on its essential class  
 \begin{enumerate}[(a)]
     \item $\left(\mathbb E[\xi_1^-]<\mathbb E[\xi_1^+],\, \mathbb E[(\xi_1^+)^p]<+\infty \right)$ and $\left(\mathbb E[(\xi_1')^+]<\mathbb E[(\xi_1')^-],\, \mathbb E[(\xi_1'^{-})^q]<+\infty \right)$;
      
     \item $\left(\mathbb E[\xi_1^-]=\mathbb E[\xi_1^+],\, \mathbb E[(\xi_1^+)^{1+p}]<+\infty \right)$ and $\left(\mathbb E[(\xi_1')^+]=\mathbb E[(\xi_1')^-],\, \mathbb E[(\xi_1'^{-})^{1+q}]<+\infty \right)$; 
      
     \item $\left(\mathbb E[\xi_1^-]<\mathbb E[\xi_1^+],\, \mathbb E[(\xi_1^+)^p]<+\infty \right)$ and $\left(\mathbb E[(\xi_1')^+]=\mathbb E[(\xi_1')^-],\, \mathbb E[(\xi_1'^{-})^{1+q}]<+\infty \right)$.\\
     The similar condition holds when swapping the roles of $\xi_1$ and $\xi_1'$. 
      
      
 \end{enumerate} 
 \end{theo}
 \begin{proof}
 As mentioned above, it remains to check that $\mathbb E[(Y_n)^p]<+\infty$ and $\mathbb E[(-Y_n')^q]<+\infty$. The set of conditions $(a)$ means that the chain moves with positive drift on the left and negative drift on the right while $(b)$ represents the center case which was already done by \textsc{Chow} and \textsc{Lai} (see \cite{ChowLai}).  The others are partly mixed from both of $(a)$ and $(b)$, so we will leave the proof only for the first case. 
 
 Notice that $0<p<1$ and by the Wald's identity, we obtain
     \begin{align*}
       \mathbb E[(\xi_1+ \ldots+\xi_{t_1})^p \mid X_0^{(0)}=x<0] &\leq  \mathbb E[({\xi_1^+}+ \ldots+{\xi_{t_1}^+})^p\mid X_0^{(0)}=x<0]\\
       &\leq \mathbb E[(\xi_1^+)^p+\ldots +(\xi_{t_1}^+)^p\mid X_0^{(0)}=x<0]\\
       &= \mathbb E[(\xi_1^+)^p] \, \mathbb E[t_1\mid X_0^{(0)}=x<0]<+\infty
     \end{align*}
     due to $\mathbb E[t_1 \mid X_0^{(0)}=x<0]<+\infty$. Indeed, if $\mathbb E[\xi_1^+]<+\infty$ then $\mathbb E[\xi_1]>0$ and $\mathbb E[\vert \xi_1\vert]<+\infty$ and the \textsc{Feller'}s result tells us that $\mathbb E[t_1\mid X_0^{(0)}=x <0]<+\infty$ (see \cite{Feller}). On the other hands, if $\mathbb E[\xi_1^+]=+\infty$ then there is $L>0$ s.t. $\xi_n^{(L)}=\min\{\xi_n, L\}$ (which has finite first moment) satisfies $\mathbb E[\xi_n^{(L)}]=\mathbb E[\xi_1^{(L)}]>0$. The first ascending ladder time $t_1^{(L)}$ associated with $S_n^{(M)}= \xi_1^{(L)}+\ldots + \xi_n^{(L)}$ has finite expectation by what we just said. Therefore, $t_1$ is integrable since $t_1 \leq t_1^{(L)}$. The argument showing $\mathbb E[(-\xi_1'-\ldots -\xi_{t_1}')^q \mid X_0^{(0)}=x\geq 0] <+\infty$ goes exactly the same line. 
 \end{proof}
 \begin{remark}
  As a direct consequence of the above proof and Corollary \ref{positive-recurrent}, we deduce that when $\mathbb E[\xi_1^-]<\mathbb E[\xi_1^+]<+\infty$ and $\mathbb E[(\xi_1')^+]<\mathbb E[(\xi_1')^-]<+\infty$ then the process $\mathcal{X}^{(0)}$ is positive recurrent. Indeed, in this case, $\mathbb{E}[\tau^{(0)}] <+\infty$ and the result follows by a classical theorem of induced processes. 
 \end{remark} 
  \subsection{Recurrence of $ \mathcal{X}^{(\alpha)}$ with $0\leq \alpha\leq 1$}
 We end this section by proving our main result 
 
\begin{coro} \label{recurrence_alpha} If at least one of the assumptions of Theorem \ref{extend_condition} is satisfied, then the general oscillating process $\mathcal X^{(\alpha)}(\mu, \mu')$ is recurrent.
\end{coro}
 \begin{proof}
   It is clear that $\mathcal{X}^{(0)}$ and $\mathcal{X}^{(1)}$ (suitably modified) are recurrent. Now, we have
 \begin{align*}
     \mathbb{P}_0[\uptau^{(\alpha)}=n] &=\mathbb{P}_0  [X_1^{(\alpha)}\neq 0, X_2^{(\alpha)} \neq 0,\ldots,X_{n-1}^{(\alpha)} \neq 0, X_n^{(\alpha)} =0 ]\\
     &= \displaystyle \sum_{k=1}^{+\infty} \mathbb{P}_0[X_1^{(\alpha)}=k]\,\mathbb{P}_k  [X_1^{(\alpha)}\neq 0, X_2^{(\alpha)} \neq 0,\ldots,X_{n-1}^{(\alpha)} \neq 0, X_{n-1}^{(\alpha)} = 0 ] \\
     &+ \displaystyle \sum_{k=-\infty}^{-1} \mathbb{P}_0[X_1^{(\alpha)}=k]\,\mathbb{P}_k  [X_1^{(\alpha)}\neq 0, X_2^{(\alpha)} \neq 0,\ldots,X_{n-1}^{(\alpha)} \neq 0, X_{n-1}^{(\alpha)} = 0]\\
     &= \alpha\, \underbrace{\displaystyle \sum_{k\neq 0} \mu(k)\,\mathbb{P}_k[\uptau^{(1)}=n-1]}_{\mathbb{P}_0[\uptau^{(1)}=n]} +(1-\alpha)\,\underbrace{\displaystyle \sum_{k\neq 0} \mu'(k)\,\mathbb{P}_k[\uptau^{(0)}=n-1]}_{\mathbb{P}_0[\uptau^{(0)}=n]}. 
 \end{align*}
 Summing over $n\geq 1$ in both sides, it readily implies $\mathbb{P}_0[\uptau^{(\alpha)}<+\infty]=1$ as expected.
 
 \end{proof}

 \end{document}